\documentclass{amsart}

\usepackage{amssymb}
\usepackage{comment}
\usepackage{mathrsfs}
\usepackage[stretch=10,shrink=10]{microtype}
\usepackage[showframe=false, margin = 1.5 in]{geometry}
\usepackage{mathtools}
\usepackage{xifthen}
\usepackage[shortlabels]{enumitem}
\usepackage{ifthen}
\newcounter{dummy}
\usepackage{xcolor}
\usepackage{stmaryrd}
\usepackage{tikz,tikz-cd}
\usetikzlibrary{decorations.pathmorphing}
\usepackage{array}
\usepackage{color}
\usepackage{colortbl}
\usepackage{hyperref}
\usepackage[normalem]{ulem}
\hypersetup{
     colorlinks   = true,
     citecolor    = black
}
\usepackage{theoremref}

\makeatletter
\def\thmref@flush{%
   \ifx\thmref@last\empty\else
      \ifthmref@comma, \thmref@finaltrue\fi \thmref@commatrue
      \thmref@last \ifx\thmref@stack\empty\else s\fi \thmref@num 0
      \let\do\thmref@one \thmref@stack
      \ifcase\thmref@num\or\space and\else\thmref@finaltrue, and\fi
      ~\ref{\thmref@head}\let\thmref@stack\empty\fi}
\def\thmref@one#1{\ifnum\thmref@num>0,\fi
   \space\ref{#1}\advance\thmref@num 1\relax}
\makeatother

\makeatletter
\newcommand\myitem[1][]{\item[#1]\refstepcounter{dummy}\def\@currentlabel{#1}}
\makeatother

\newcommand{\E}{\mathbf{E}}

\renewcommand{\P}{\mathbf{P}}

\newcommand{\1}{\mathbf{1}}

\DeclareMathOperator{\Poi}{Poi}

\DeclareMathOperator{\Geo}{Geo}

\DeclarePairedDelimiter\abs{\lvert}{\rvert}%
\DeclarePairedDelimiter\tvnorm{\lVert}{\rVert_{\mathrm{TV}}}% could add subscript tv or not
\DeclarePairedDelimiter\ceil{\lceil}{\rceil}%
\DeclarePairedDelimiter\ii{\llbracket}{\rrbracket}%

\newcommand{\erase}{\mathsf{GreedyPath}}

\newcommand{\ZZ}{\mathbb{Z}}

\newcommand{\NN}{\mathbb{N}}

\newcommand{\Instr}{\mathsf{Instr}}
\newcommand{\eosos}[1]{\mathcal{O}_{#1}}
\newcommand{\ip}[1]{\mathcal{I}_{#1}}
\newcommand{\s}{\mathfrak{s}}

\newcommand{\critical}{\rho}
\newcommand{\lpcritical}{\rho_*}

\newcommand{\smax}{s^{\max}}

\newcommand{\rt}[2]{\mathcal{R}_{#2}(#1)}
\newcommand{\lt}[2]{\mathcal{L}_{#2}(#1)}
\newcommand{\emax}{e_{\max}}

\DeclareFontEncoding{LS1}{}{}
\DeclareFontSubstitution{LS1}{stix}{m}{n}
\DeclareSymbolFont{stixletters}{LS1}{stix}{m}{it}
\DeclareMathAccent{\cev}{\mathord}{stixletters}{"91}
\DeclareMathAccent{\vec}{\mathord}{stixletters}{"92}
\DeclareMathAccent{\vecev}{\mathord}{stixletters}{"95}
\newcommand{\greedy}[1][k]{\text{$k$-\erase}}
\makeatletter

\newcommand{\crist}[1][\@nil]{%
\def\tmp{#1}%
   \ifx\tmp\@nnil
       \lpcritical
    \else
       \lpcritical^{(#1)}
    \fi}
\makeatother

\newcommand{\Left}{\ensuremath{\mathtt{left}}}
\newcommand{\Right}{\ensuremath{\mathtt{right}}}
\newcommand{\Sleep}{\ensuremath{\mathtt{sleep}}}
\newtheorem{thm}{Theorem}[section]
\newtheorem{lemma}[thm]{Lemma}
\newtheorem{prop}[thm]{Proposition}
\newtheorem{cor}[thm]{Corollary}

\newtheorem*{conjecture*}{Density Conjecture}

\theoremstyle{remark}
\newtheorem{remark}[thm]{Remark}
\theoremstyle{definition}

\newcommand{\HOX}[1]{\marginpar{\scriptsize #1}}
\definecolor{hancolor}{rgb}{0.0 0.0, 1.0}
\newcommand{\critFE}{\critical_{\mathtt{FE}}}
\newcommand{\critDD}{\critical_{\mathtt{DD}}}

\newcommand{\mo}{\mathfrak{m}}
\newcommand{\restrictedOp}[2]{#1\if\relax\detokenize{#2}\relax\else^{#2}\fi}
\newcommand{\restrictedSfOp}[2]{\restrictedOp{\mathsf{#1}}{#2}}
\newcommand{\A}[1][]{\restrictedSfOp{A}{#1}}
\newcommand{\Stab}[1][]{\restrictedSfOp{S}{#1}}

\newcommand{\Jump}[1][]{\restrictedSfOp{J}{#1}}

\newcommand{\Topple}[1][]{\restrictedSfOp{T}{#1}}
\newcommand{\odom}{f}
\newcommand{\odomm}{g}

\newcommand{\trueodom}{\odom_{\Stab}}
\newcommand{\config}{\sigma}
\newcommand{\state}{(\config, \odom)}
\newcommand{\configg}{\tau}
\newcommand{\configstart}{\config_0}

\newcommand{\V}{V}

\newcommand{\onepersite}{1_\V}
\newcommand{\zeroconfig}{0}
\newcommand{\sweepers}{\config_{\mathsf{sweep}}}
\newcommand{\drive}{\config^{\mathsf{drive}}}
\newcommand{\sweepersstart}{\sweepers}
\newcommand{\cars}{\config_{\mathsf{car}}}

\newcommand{\carsstart}{\cars}
\newcommand{\stationarysymb}{\pi}
\newcommand{\stablelaw}[1]{\mu({#1})}

\newcommand{\Tfull}{T_{\mathtt{full}}}
\makeatletter
\newcommand{\tmix}[1][\@nil]{%
\def\tmp{#1}%
   \ifx\tmp\@nnil
       t_{\mathtt{mix}}
    \else
       t_{\mathtt{mix}}^{(#1)}
    \fi}
\newcommand{\tsep}[1][\@nil]{%
\def\tmp{#1}%
   \ifx\tmp\@nnil
       t_{\mathtt{sep}}
    \else
       t_{\mathtt{sep}}^{(#1)}
    \fi}    
\makeatother

\newcommand{\supp}{\operatorname{supp}}

\newcommand{\m}{\mathfrak{m}}
\begin{document}

 \title{Cutoff for activated random walk}  

 \author{Christopher Hoffman}
 \address{Christopher Hoffman, Department of Mathematics, University of Washington}
 \email{\texttt{choffman@uw.edu}}
 \author{Tobias Johnson}
 \address{Tobias Johnson, Departments of Mathematics,  College of Staten Island, City University of New York}
\email{\texttt{tobias.johnson@csi.cuny.edu}}
 \author{Matthew Junge}
 	\address{Matthew Junge, Department of Mathematics, Baruch College, City University of New York}
	\email{\texttt{matthew.junge@baruch.cuny.edu}}
 \author{Josh Meisel}
 	\address{Josh Meisel, Department of Mathematics, Graduate Center, City University of New York}
	\email{\texttt{jmeisel@gradcenter.cuny.edu}}

\begin{abstract}
We prove that the mixing time of driven-dissipative activated random walk on an interval of length $n$ with uniform or central driving exhibits cutoff at $n$ times %$\critFE n$ with $\critFE$ 
the critical density for activated random walk on the integers. 
The proof uses a new result for arbitrary graphs showing that the chain is mixed once activity is likely at every site.
\end{abstract}
  \maketitle

\section{Introduction}

Avalanches, earthquakes, solar flares, and many other natural systems display power-law magnitudes and fractal structures akin to laboratory-tuned phase transitions \cite{watkins2016twentyfive}. Bak, Tang, and Wiesenfeld theorized that the gradual accumulation and sudden dissipation of energy among many spatial degrees of freedom cause such systems to self-organize into a critical-like state \cite{BakTangWiesenfeld87, BakTangWiesenfeld88}. They proposed the \emph{abelian sandpile} as a mathematical model. While the model exhibits intricate
structure \cite{levine2016apollonian, levine2017apollonian}, it is now believed that
slow mixing makes it sensitive to minor modifications, giving rise to nonuniversal behavior disqualifying it as a model of self-organized criticality \cite{fey2010driving, levine2015threshold}.

\emph{Activated random walk} (ARW) is a related stochastic model of self-organized criticality
that is believed to exhibit the desired universality. It can be formulated as an interacting particle system on a graph with active and sleeping  particles. Active particles perform simple random walk at exponential rate 1. When an active particle is alone, it falls asleep at exponential rate $\lambda \in (0,\infty)$. Sleeping particles remain in place but become active if an active particle moves to their site. If all active particles fall asleep, the system stabilizes with every site in the final configuration either empty or containing exactly one sleeping particle.

We consider the \emph{driven-dissipative} version, 
in which particles are injected into a finite graph and killed at the boundary.
Formally, the process takes place on a finite graph
$G= (V\cup\{z\},E)$ where the vertex $z$ represents a sink.
Typically, our graph will be the interval $\ii{1,n}:=[1,n]\cap\ZZ$
with the sink $z$ equal to the outer boundary vertices $\{0,n+1\}$.
Fix some \emph{driving sequence} $\mathbf u = (u_t)_{t \geq 1}$ of vertices in $V$, 
which can be deterministic or random but must be independent of the underlying particle dynamics (we make this restriction precise in Section~\ref{sec:odometers}).
Most interest has focused on \emph{uniform driving}, where each $u_t$ is selected
independently and uniformly from $V$, and \emph{central driving}, where $u_t$ is deterministically
equal to the vertex at the center of $V$ (on $\ii{1,n}$, we take $u_t=\ceil{n/2}$).
Specify an initial configuration of particles $\sigma_0\in\{\s,0,1,\ldots\}^V$,
where $\s$ represents a sleeping particle and $k\in\NN:=\{0,1,\ldots\}$ represents a quantity
of active particles.
At each step $t\geq 1$, an active particle
is added at $u_t$. Then, we \emph{stabilize} the system, meaning that we let it
evolve according to ARW dynamics until it arrives
at a configuration $\sigma_t$ containing only trapped and sleeping particles.
The sequence $\sigma_t\in \{0,\s\}^V$ for $t\geq 1$ is a Markov chain
on the state space of stable configurations on $V$.

Driven-dissipative activated random walk is believed to converge in time to a critical state
with power-law spatial correlations \cite[Section~1.3]{rolla2020activated}.
As we will explain in the following section, it is of great importance to establish how quickly this convergence occurs.
Levine and Liang established that with uniform or central driving,
driven-dissipative ARW on the Euclidean ball of volume $n$ 
is mixed after the addition of $(1+o(1))n$
particles \cite[Corollary~5]{levine2021exact}, 
in contrast with the abelian sandpile whose mixing time (in dimension two,
at least) contains an extra logarithmic factor \cite{hough2019sandpiles,hough2021cutoff}.
In support of the universality of ARW, they conjectured that with uniform driving 
it mixes as quickly as possible, i.e., 
as soon as the system density reaches a hypothesized 
density $\critDD=\critDD(\lambda,d)$ 
equal to the limiting density of the stationary state of the Markov chain
as the ball size is increased \cite[Conjecture~10]{levine2021exact}.

Levine and Liang further proposed that
$\critDD$ coincides with a critical density $\critFE$
for a phase transition for \emph{fixed-energy ARW} on $\ZZ^d$,
reiterating the \emph{density conjecture} first proposed by Dickman, Mu\~noz, Vespignani, and Zapperi
\cite{dickman1998self}.
Fixed-energy ARW is initialized with a configuration of active particles
of given density $\rho$, which then remains fixed as the system evolves according to the ARW
dynamics with no addition or destruction of particles (see \cite{rolla2020activated}).
This system has a traditional absorbing-state phase transition
in the parameter $\rho$ at critical density
$\rho_{\texttt{FE}}= \rho_{\texttt{FE}}(\lambda, d)$:
for any ergodic initial configuration of density $\rho$,
at each site activity eventually ceases almost surely if $\rho < \critFE$,
while activity persists for all time almost surely if $\rho>\critFE$
\cite{rolla2019universality}. 
Three of the authors of the present paper recently proved the density conjecture
in dimension one, establishing the existence of $\critDD$ and its equality with $\critFE$
\cite{hoffman2024density}.

Now, we establish Levine and Liang's mixing time conjecture in dimension one.
% \HOX{The ``Thus" here seems off to me. I think this is explicitly part (3) of their conjecture? Like they conjecture the mixing time is $\critDD$ in (1) and (2) and that $\critDD=\critFE$ in (3) -jm}%
% \HOX{I revised this. It's tricky to say it well, because morally speaking Levine and Liang were proposing
% cutoff at the critical density, and then they were stating the density conjecture
% to explain what they meant by ``the critical density''. I've tried to explain now. --TJ}
To state the result, let $\pi=\pi(G,\lambda)$ denote the stationary distribution of 
the driven-dissipative chain for a given graph $G$ and
sleep rate $\lambda\in(0,\infty)$. Note that $\pi$ does not depend on the
driving sequence $\mathbf{u}$ \cite[Lemma~6]{levine2021exact}.
Denote the total-variation norm by $\tvnorm{\cdot}$. 
\begin{thm} \thlabel{thm:mix}
  Let $(\sigma_t)_{t\geq 1}$ denote the driven-dissipative ARW chain on $\ii{1,n}$
  with uniform or central driving. For any sleep rate $\lambda\in(0,\infty)$ and $\epsilon>0$,
  there exist constants $c,C>0$ such that for all $n \geq 1$
  \begin{alignat}{2}
    \max_{\sigma_0}\tvnorm{\P(\sigma_t\in\cdot)-\pi}&\leq Ce^{-cn},&\quad\text{for $t>(\critFE+\epsilon)n$},
    \label{eq:mixing.upper}\\
    \intertext{while with $\sigma_0$ equal to the empty configuration,}
    \tvnorm{\P(\sigma_t\in\cdot)-\pi}&\geq 1-Ce^{-cn},&\quad\text{for $t<(\critFE-\epsilon)n$.}
    \label{eq:mixing.lower}
  \end{alignat}
\end{thm}
The maximum in \eqref{eq:mixing.upper} is taken over all configurations $\sigma_0\in\{\s,0,1,\ldots\}^V$.
This theorem implies that the mixing time
\begin{align*}
  \tmix[n](\epsilon) := \min\bigl\{t\colon \max_{\sigma_0}\tvnorm{\P(\sigma_t\in\cdot)-\pi}\leq\epsilon\bigr\}
\end{align*}
    satisfies $\tmix[n](\epsilon)=(\critFE+o(1))n$  and that the chain exhibits \emph{cutoff}, % at $\critFE n$,
meaning that $\tmix[n](\epsilon)/\tmix[n](1-\epsilon)\to 1$ for any $\epsilon>0$.
%In short, the chain exhibits cutoff at mixing time $\critFE n$.

% \HOX{This is over-promising, because we need activity everywhere with sufficiently high probability. Not the same as the density conjecture. Could be worth stating a conditional theorem. -MJ}
% \HOX{I edited my text to be more vague about what was necessary. My first thought is that it would be fine
% to state a conditional theorem so long as we don't waste time with a boilerplate proof.
% But it's not so easy to state a generic conditional theorem---do we assume exponential
% bounds on activity and an exponential lower bound on density at of the stationary state? Or
% just high probability bounds? Etc. I think the generic conditional result
% \thref{thm:one} plus an imprecise claim is fine. --TJ}

Our result provides evidence for the universality of ARW in several ways.
First, it shows that the initial configuration is forgotten by the system
as quickly as possible, demonstrating that it is a ``small-scale detail'' with ``negligible effect on
large-scale observables,'' to quote Levine and Silvestri's description of universality 
\cite{levine2023universality}.
% \HOX{By initial state do you mean the initial configuration? But then is this just true of all Markov chains? -jm}\HOX{I added in the ``as quickly as possible'' part. Yes, the initial
% configuration is forgotten in any Markov chain. But the idea here is that the faster it's doing,
% the less that the ``small-scape detail'' actually matters. --TJ}
Next, numerical and theoretical results on the abelian
sandpile model show that slow mixing is at the heart of its nonuniversality, as we discuss
further in Section~\ref{sec:sandpile.history}.
ARW does not exhibit this flaw and can serve as a universal model
for self-organized criticality.

Our techniques for bounding the mixing time apply on any graph. 
Given an extension of \cite{hoffman2024density}
to higher dimensions---a substantial
task, admittedly---we could establish Levine and Liang's conjecture in all dimensions.
We explain further in Section~\ref{sec:methods}.

In the rest of this section, we give a fuller account of self-organized criticality
and ARW, with an emphasis on the importance of fast mixing
for a model of self-organized criticality. We also describe our techniques and
how they differ from previous ones.

\subsection{Self-organized criticality and sandpiles}\label{sec:sandpile.history}

Bak, Tang, and Wiesenfeld's original and simplest illustration of self-organized criticality is a growing sandpile on a table \cite{BakTangWiesenfeld87}. As sand is sprinkled from above, the pile's slope steadily increases until it reaches a critical slope that is maintained as sand spills off the table in avalanches. They posited that this exemplifies how many complex systems, such as snow slopes, tectonic plates, and the surface of the sun, manage to stay poised in critical-like states despite no apparent external tuning.

Seeking a tractable mathematical model that drives itself
to criticality without external tuning, they proposed the abelian sandpile on a two-dimensional square lattice  \cite{BakTangWiesenfeld87}. 
Particles
remain still until four build up on a single site. Then the site \emph{topples},
sending the four particles
deterministically to the site's four neighbors and possibly triggering a further cascade.
As with ARW, the driven-dissipative version of the model is defined on a finite graph, and in each
step a particle is injected into the system with killing at the boundary.
In the picture of self-organized criticality painted by Bak, Tang, and Wiesenfeld, 
the density of the system increases steadily
as particles are added until it reaches a critical value. 
Then, the ``barely stable'' system maintains itself at criticality
by dissipations of energy with a power-law tail, as in a traditional thermodynamic
system at criticality.

As with ARW and other sandpile models, the abelian sandpile has a fixed-energy version in which
the graph is initialized with particles at a given density which then evolve according to
the sandpile dynamics with no further addition or deletion of particles.
The critical density for the fixed-energy model of the abelian sandpile on an infinite
lattice, often called the \emph{threshold density}, is defined as the supremum
over all $\rho$ such that the configuration consisting of i.i.d.-$\Poi(\rho)$ particles
per site is stabilizable (see \cite{fey2009stabilizability} for the definition
of stabilizable). The threshold density is defined on finite graphs without boundary
as the expected density when the configuration becomes nonstabilizable when adding
particles one at a time at random to the system (see for example \cite[Sections~1.3 and 1.8]{levine2015threshold}).

Dickman, Mu\~noz, Vespignani, and Zapperi theorized
that self-organized criticality is a consequence of driven-dissipative systems naturally
converging to the corresponding fixed-energy system at criticality 
\cite{dickman1998self,vespignani2000absorbing}.
We can think of their theory as proposing universality for sandpiles at criticality,
in that the critical state obtained by traditional tuning or by self-organization has the same
properties.
The most basic prediction of this theory is that
the density in the driven-dissipative chain converges to the critical density of the fixed-energy system.
This density conjecture was originally claimed for the stochastic
sandpile model \cite{dickman1998self} and then the abelian sandpile \cite{vespignani2000absorbing},
but it was believed to hold for ARW as well \cite{rolla2020activated}.
Combining Dickman et al.'s theory with Bak, Tang and Wiesenfeld's original picture, we expect a suitably
universal sandpile model to behave as follows: the driven-dissipative system increases in
density linearly until it arrives at the fixed-energy critical density and remains there.

The behavior of the abelian sandpile on the two-dimensional lattice has has never been rigorously
established, but Fey, Levine, and Wilson have presented convincing numerical and theoretical
evidence that it does not fit the above picture.
First, their simulations show that the stationary density of the driven-dissipative system on a box
of width~$n$ and the threshold density on a torus of width~$n$ do not converge to the same value, though the limiting densities differ only in the fourth decimal place \cite{fey2010driving}.
(The threshold density on large tori are believed to converge to that of the infinite lattice, though
this too remains unproven.)
On certain more tractable families of graphs, Fey, Levine, and Wilson rigorously
determined the stationary and threshold densities. In all cases they analyzed,
they found that the driven-dissipative chain increased linearly in density until
arriving at the threshold density but then slowly evolved from there to a different limit \cite{fey2010approach}.
This suggests that a more rapidly mixing model would exhibit more universal behavior.

Another result of Levine's shows that a different version of the threshold density has better properties.
Instead of placing particles one at a time on an empty graph until reaching
a nonstabilizable configuration, we place them on a graph that starts with a hole of depth $s$
at every vertex. A vertex must have its hole entirely filled before it can start to accumulate
particles and topple. On a single fixed graph, the threshold density starting from
this landscape converges as $s\to\infty$ to the stationary density of the driven-dissipative system
\cite{levine2015threshold}. 
This result also addresses another failure of universality,
namely that that the threshold density of the abelian sandpile on $\ZZ^2$
depends on the precise distribution of the initial configuration, in the sense that there exist ergodic
initial configurations at densities $\rho_1<\rho_2$ such that the model is stabilizable
at density~$\rho_2$ but not at density $\rho_1$ \cite{fey2005organized}.
Levine's result demonstrates that the nonuniversal behavior of the threshold density
can be corrected by allowing the system more time to mix.

In search of a universal model of self-organized criticality, attention turned to two sandpile models
with stochastic dynamics. First is the stochastic sandpile model, in which multiple particles on the same vertex
jump independently to random neighbors while lone particles remain stationary \cite{manna1991two}.
The second is ARW, introduced as a more tractable alternative.
These two models have so far been found to exhibit
similar behavior, but much more has been proven for ARW;
see \cite{rolla2020activated} for the standard reference.

Initial results on ARW focused on its fixed-energy version on $\ZZ^d$.
Even for $d=1$, there is no simple argument demonstrating the existence of either a nontrivial
fixation phase or activity phase. Considerable effort has gone into proving the existence
of both phases, culminating in a complete proof that $0<\critFE<1$ for all $\lambda>0$ and $d\geq 1$ \cite{rolla2012absorbing, basu2018nonfixation,SidoraviciusTeixeira17,StaufferTaggi18,HoffmanRicheyRolla20,hu2022active,forien2022active,asselah2024critical}. This line of research along with \cite{BasuGangulyHoffmanRichey19}
demonstrated that fixed-energy ARW on tori also exhibits a phase transition in initial density.
Other notable results include
universality for the critical density over all ergodic initial configurations of active
particles \cite{rolla2019universality}, recently generalized further in dimension one
\cite{brown2024activated}; and continuity in $\lambda$ of the critical density $\critFE(\lambda,d)$
\cite{taggi2023essential}.

For driven-dissipative ARW, less is known. One line of research to be discussed in the following section
has investigated its mixing properties on general graphs \cite{levine2021exact, bristiel2022separation}.
J\'arai, M\"onch, and Taggi analyzed driven-dissipative ARW in the mean-field setting---on a complete
graph on $n$ vertices plus a sink vertex---and found the exact limiting 
density $\rho_c$ for the stationary distribution.
They also determined a precise window of order $\sqrt{\log n/n}$ slightly above $\rho_c$ that contains
the density of a sample from the stationary distribution with probability tending to $1$ \cite{jarai2023critical}.
The remaining results we know of are for an interval in $\ZZ$.
Forien proved in this setting that starting from 
a configuration of density greater than $\critFE$, a step of the driven-dissipative chain in dimension one
has positive probability of losing a macroscopic quantity of particles \cite{forien2024macroscopic}, 
progress toward the
density conjecture. Shortly after, the density conjecture 
was proven in dimension one using different methods \cite{hoffman2024density}. Next, it was shown that the density of
the driven-dissipative chain in dimension one increases linearly to $\critFE$ and then remains
there \cite{hoffman2024proof}, consistent with
Bak, Tang, and Wiesenfeld's original vision of a self-organizing
sandpile reaching and then sustaining a critical state.
\thref{thm:mix} gives a stronger confirmation of this vision:
starting from any initial configuration, 
not only is the critical density maintained from time $\critFE n$ on, 
but the critical state has been reached and is maintained.

\subsection{Mixing time for ARW: prior results and methods}
Levine and Liang initiated the study of the mixing time of the driven-dissipative chain
in \cite{levine2021exact}, proving upper bounds of $(1+o(1))n$
on a Euclidean ball in $\ZZ^d$ of volume~$n$ with uniform
or central driving. The bound is uniform in the sleep rate $\lambda$ and a more precise
estimate is given for the $o(1)$ factor depending on $d$. Their result starts with the important
observation that the stationary distribution of the chain can be exactly sampled
by stabilizing the driven-dissipative system starting from any configuration
$\sigma_0$ containing at least one active particle at every site \cite[Theorem 1]{levine2021exact}.
We sketch the proof now. The well-known \emph{abelian property} for ARW (to be described
rigorously in Section~\ref{sec:odometers}) states that if we represent the system
as particles executing random instructions placed on the vertices of the graph,
the final state of the system after stabilization is the same regardless of the order
of topplings.
By this property, if we run the driven-dissipative chain starting with $\sigma_0$, 
then $\sigma_t$ can be viewed as the result
of stabilizing the configuration consisting of $\sigma_0$ with extra particles
added at sites $u_1,\ldots,u_t$, where
$\mathbf{u}=(u_1,u_2,\ldots)$ is the driving sequence. Applying the abelian property
again, we can stabilize this configuration by toppling sites to move the extra particles around.
Since $\sigma_0$ contains a particle at all sites, these extra particles never fall asleep
and eventually reach the sink. Now we are reduced to stabilizing $\sigma_0$, thus proving
that $\sigma_t$ is distributed as the stabilization of $\sigma_0$ for all $t\geq 1$.
Thus the chain is already at stationarity after one step. As a corollary,
this argument confirms that the stationary distribution does not depend on the driving
sequence, since it is always equal to the stabilization of $\sigma_0$. 

Levine and Liang then use this fact to bound the mixing time of ARW in terms of
\emph{internal diffusion-limited aggregation} (IDLA), which is the $\lambda=\infty$ version of ARW in which particles fall asleep immediately when alone.
The idea is as follows. Suppose we stabilize some configuration $\sigma$. By the abelian property,
we can do it in two steps: First, topple sites in any order until all particles are alone
or at the sink; call the resulting state $\sigma'$. Next, stabilize $\sigma'$. The first step
can be thought of as running the system with IDLA dynamics.

If the first step of toppling yields $\sigma'\equiv 1$, then the second stage puts it at stationarity.
Hence, once the chain has been driven long enough so that $\sigma_0+\delta_{u_1}+\cdots+\delta_{u_t}$
spreads to the entire graph under IDLA dynamics, it is mixed. 
Defining the \emph{fill time} $\Tfull$ as the first time this occurs, 
the mixing time can be bounded in terms of $\Tfull$  \cite[eq.~(10)]{levine2021exact}. 
The proof is complete once it is shown that
that $\Tfull$ is unlikely to exceed $(1+o(1))n$ where $n$ is the volume
of the graph; in many circumstances $\Tfull$ is in fact known very precisely
\cite{jerison2012logarithmic,jerison2013internal}.

Bristiel and Salez continue with this line of analysis and apply it to the \emph{separation mixing time}
\cite{bristiel2022separation}.
To define this mixing time, let $P^t(\sigma,\sigma')$ denote the probability that the driven-dissipative
chain starting from configuration $\sigma$ transitions to configuration $\sigma'$ in step~$t$.
The \emph{separation mixing time} of the chain on a given graph is
\begin{align}\label{eq:sep.dist}
  \tsep(\epsilon):= \min\biggl\{ t\colon \max_{\sigma,\sigma'}\biggl(1 - \frac{P^t(\sigma,\sigma')}{\pi(\sigma')}\biggr)<\epsilon\biggr\},
\end{align}
where $\pi$ is the stationary distribution.
The separation mixing time is an upper bound on the usual mixing time
\cite[Lemma~6.16]{levin2017markov}. Levine and Liang's argument in fact bounds $\tsep$,
since $\Tfull$ is a strong stationary time. Bristiel and Salez observe that to stabilize
$\sigma_0+\delta_{u_1}+\cdots+\delta_{u_t}$ and arrive at the configuration $\sigma'\equiv\s$,
it \emph{must} happen that $\Tfull\leq t$. This observation gives a lower bound on
$\frac{P^t(\sigma,\sigma')}{\pi(\sigma')}$ and leads to the conclusion that $\Tfull$
is \emph{separation optimal} and determines the exact value of the separation distance,
demonstrating that $\tsep=(1+o(1))n$ on a Euclidean ball of volume~$n$. 
Bristiel and Salez also compute the relaxation time of the
chain (the inverse of its spectral gap), which they use to establish cutoff for
the separation mixing time.

To find the separation mixing time of the driven-dissipative chain,
the analysis of \cite{levine2021exact,bristiel2022separation} is enough,
since the IDLA dynamics accurately represent the system in the extreme case of \eqref{eq:sep.dist}
where $\sigma\equiv 0$ and $\sigma'\equiv \s$. But to measure the total variation
from stationarity, reflecting the ability to couple the system with its stationary state,
this analysis via IDLA is insufficient.

\subsection{Methods of our proof}\label{sec:methods}

We establish a new criterion for mixing on a general graph: if enough particles have been added
so that all sites are visited under the ARW dynamics, then the system is mixed.
To state this result in full generality, we allow driven-dissipative ARW to take
place on any finite graph $G=(V\cup\{z\},E)$, and we allow particles
to move when active as continuous-time random walks with an arbitrary transition kernel $Q$
that we call the \emph{base chain}. The sink $z$ must be an absorbing state for $Q$, i.e., $Q(z,z)=1$,
and we also require that all vertices in $V$ are mutually accessible and that $z$
is accessible from $V$ under $Q$. 

Under the above assumptions,
driven-dissipative ARW on $G$ has a unique stationary distribution $\pi$
that does not depend on the driving sequence $\mathbf{u}$ \cite[Lemma~6]{levine2021exact}.  For a configuration $\sigma\in \{\s,0,1,\ldots\}^V$, we say that \emph{$\sigma$ visits all sites}
if all vertices in $V$ contain an active particle at some point during the stabilization of $\sigma$. Denote the law of the stabilization of $\sigma$ by $\stablelaw{\config}$.  Also, consider a system of independent random walks starting from each vertex in $V$ that all jump simultaneously at time steps $1,2,\hdots$ according to $Q$, and let $H$ be the first time at which all walks are at the sink.

\begin{thm}\thlabel{thm:one}
 Using the above notation,
 for any configuration $\sigma$ and integer $m \geq 1$ 
  \begin{align}\label{eq:one}
    \tvnorm{\stablelaw{\config}-\pi} &\leq m \P(\text{$\sigma$ does not visit all sites})  + \P(H \ge m).
  \end{align}  
\end{thm}
One should think of the right-hand side of \eqref{eq:one} as essentially $p:=\P(\text{$\sigma$ does not visit all sites})$, with $m$ and $\P(H\geq m)$ negligible.
In our applications, $p$ will be exponentially small,
$m$ polynomial, and $\P(H \geq m)$ exponentially small as functions of $\abs{V}$.

We prove this theorem with the help of what we call the \emph{preemptive abelian property} (\thref{lem:pap}).
We cannot state the property formally without the notation in Section~\ref{sec:odometers},
but the idea is that if a sleeping particle will eventually be visited, we may preemptively wake it
without altering the final stabilized state of the system. 
On first encounter, this may appear of limited use since only with foreknowledge of the stabilization do we obtain extra freedom in computing the stabilization. While this may be true from a deterministic perspective, we combine the property with probabilistic guarantees to prove our \emph{street-sweeper lemma} (\thref{lem:street-sweeper}). 
The name is inspired by New York City regulations
that prevent a car from parking in the present because a street sweeper may pass by in the
indefinite future. Start with a collection of particles called ``street sweepers'' and add additional ``car" particles. The lemma says that if ARW restricted to street sweepers visits every site, then ARW with only street sweepers can be coupled to give the same final configuration as ARW with both street sweepers and cars,
because the preemptive abelian property allows the cars to move without parking until they
drive off the graph. 

To obtain \thref{thm:one}, we place a car at every site and street sweepers at the sites of uniform or centralized driving. By Levine and Liang’s exact sampling theorem, the stabilization involving both the street sweepers and cars has exactly the stationary distribution. By the street-sweeper lemma, so long as the street sweepers are likely to visit every site, disregarding the cars and stabilizing only the street sweepers results in the same law i.e., exactly stationary.
The $m$ and $\P(H\geq m)$ error terms come from a union bound over the number of steps at which the street sweepers may fail to topple the cars. A powerful, and perhaps on first encounter surprising, aspect of the bound is that we need not guarantee multiple site visits from the street sweepers. Using a Markov property, the street sweepers can remain in place while the high probability threat of at least one visit to each site is recycled, paying a $p$ penalty with each application.

\thref{thm:one} lets us bound the mixing time in terms of the density of particles that
must be added to a system to ensure all sites are visited. This density should be $\critFE$, since it is believed that
ARW starting from any density above $\critFE$ should behave in a supercritical manner
with activity everywhere. We prove \thref{thm:mix} via \thref{thm:one} and a simple sufficient condition for an initial configuration to produce activity at a specified site (\thref{prop:exit.criterion}). A reader who understands the \emph{layer percolation} process introduced in \cite{hoffman2024density} and its correspondence to ARW will see \thref{prop:exit.criterion} as a fairly straightforward application.
Nonetheless, we view the proof, which transforms the question of ARW activity at a site to one about the growth of critical Galton-Watson processes, as a nice illustration of layer percolation's usefulness.
In higher dimensions, a mixing time bound will have to wait for a proof
that a density of particles above $\critFE$ visits all sites with high probability,
a result which would likely accompany any resolution of the density conjecture.

\subsection{Organization}
In Section~\ref{sec:odometers}, we describe the sitewise representation of ARW and introduce further notation. Section~\ref{sec:general}  contains the preemptive abelian property, street-sweeper lemma, and \thref{thm:one}. Section~\ref{sec:dim1} proves \thref{prop:exit.criterion} and derives \thref{thm:mix} as a consequence.

\section{Notation for the sitewise construction}\label{sec:odometers}
We give the \emph{sitewise construction} of ARW, which makes it an abelian network.
See \cite[Section~2.2]{rolla2020activated} for another reference.
Though this material is standard, we take the opportunity to set our notation and define
some operators that alter the state of the ARW configuration.

To give the idea of the construction before the details, we place a stack of random instructions 
$\bigl(\Instr_v(j),\,j=1,2,\ldots\bigr)$ on each site $v\in V$.
We will specify the distribution of the instructions eventually, but for now the instructions
should be taken as arbitrary deterministic sequences.
Each instruction tells a particle
on that site to try to sleep or to jump to a certain neighbor. 
Then ARW is constructed as a continuous-time Markov chain in which
active particles execute the top instruction
on their stack at rate~$1$. When an instruction is executed, we say that the its site is \emph{toppled}.
The \emph{odometer} of a sequence of topplings is the function on $V$ counting the number of topplings
at each site.
A configuration $\sigma$ is stable on a set $W$ if it contains no active particles there,
i.e., $\sigma(v)\in\{0,\s\}$ for $v\in W$.

The reason the sitewise representation is so useful is its \emph{abelian property}:
all sequences of topplings within a finite set of vertices $W$ that leave a configuration
stable on $W$ have the same odometer \cite[Lemma~2.4]{rolla2020activated}. This means that the final state
of the system depends only on the instructions and not on their order of execution.

We will represent the \emph{state} of an ARW system as a pair $(\sigma,f)$ where 
$\sigma\in\{\s,0,1,\ldots\}^V$ is a configuration
and $f\colon V\to \NN$ is \emph{the running odometer}, 
representing the topplings that have occurred so far
and specifying which instruction is next to be executed at each site. Odometer $\odom$ \textit{dominates} $\odomm$ if it does so pointwise, and \textit{strictly dominates} $\odomm$ if further there is a site $v$  where $\odom(v) > g(v)$. 
We use the convention $\abs{\s}=1$ so that $\abs{\sigma(v)}$ is the number of particles in $\sigma$
at $v$, sleeping or not. We extend the ordering of the integers as well as their addition to the value $\s$ by letting $0 < \s < 1 < \ldots$, setting $0 + \s = \s = \s + 0$, and for $k \ge \s$ letting $\s + k = k + \s = |k| + 1$. Configurations are added pointwise, so $(\config + \tau)(v) = \config(v) + \tau(v)$.
    % We define $\abs{\config} := \sum_{v \in \V}\abs{\config(v)}$. 

We define the following notation:
\begin{description}
  \item[support of a configuration] The support of a configuration $\sigma$, denoted
    $\supp \config$, is the set of sites $v\in V$ for which $\config(v) \neq 0$.
  \item[toppling operator ${\Topple[v]}$] For a configuration $\sigma$ with a particle at vertex~$v$ (i.e., satisfying
    $v\in\supp \sigma$), the operator
    $\Topple[v]$ acts on state $(\sigma,f)$ by executing
    the next instruction at $v$ and incrementing $f$ at $v$.
    That is, $\Topple[v](\sigma,f)=(\sigma',f+\delta_v)$,
    where $\sigma'$ is the configuration that results from executing the instruction
    $\Instr_v(f(v))$. For example, if $\Instr_v(f(v))$ is a jump instruction from $v$ to $w$
    and $\sigma(v)=2$ and $\sigma(w)=\s$, then $\sigma'$ is identical to $\sigma$
    except that $\sigma'(v)=1$ and $\sigma'(w)=2$.
    See \cite[Section~2.2]{rolla2020activated} for the complete formalism.
    
    If the particle at $v$ was active, then the toppling is \emph{legal}, and $\Topple[v]$ is \emph{legal} for $\state$. If the particle
    at $v$ was sleeping, then toppling site~$v$ violates the dynamics of the model, but $\Topple[v]$
    is still defined. In this case, the toppling is \emph{acceptable}, and $\Topple[v]$ is \emph{acceptable} for $\state$. 

    % For $W=\{v_1, \ldots, v_n\}$, 
  
  \item[stabilization operator {$\Stab[U]$}] The operator $\Stab[U]$ alters the state
    by legally toppling vertices in $U\subseteq V$ until the configuration is stable on $U$, i.e.,
    $\Stab[U](\sigma,f) = \Topple[v_n]\Topple[v_{n-1}]\ldots \Topple[v_1](\sigma,f)$ for any 
    legal toppling sequence of vertices in $U$ yielding a state whose configuration is stable
    on $U$.
    The operator is well defined by the abelian property. We define $\Stab=\Stab[V]$,
    the operator stabilizing the entire graph (except the sink).
    
    Let $\Stab[U](\sigma,f)=(\sigma',f')$. We call $\sigma'$, $f'$, and $(\sigma',f')$
    the \emph{stabilized configuration},
    \emph{stabilizing odometer}, and \emph{stabilized state}, respectively, of $(\sigma,f)$ on $U$.
    And in all these cases we drop the mention of $U$ when $U=V$. We also call $\odom'$ the \emph{odometer stabilizing} $\state$ (on $U$). 
    
    We often use the identity $\Stab (\Stab[U] \state) = \Stab \state$, which follows from the abelian property. In other words, $\Stab[U] \state$ has the same stabilized state as $(\sigma,f)$. 
    
    Another useful fact is that if $\sigma$ and $\sigma'$ agree on $U$, then their stabilized states
    on $U$ also agree. That is,
    for \begin{align*}
        (\configg, \odomm) &:= \Stab[U](\sigma, f)\\
        (\configg', \odomm')&: = \Stab[U](\config', f).
    \end{align*}
    Then $\odomm=\odomm'$ and $\configg\rvert_U=\configg'\rvert_U$. Outside of $U$ both configurations increase by the same number of active particles, meaning for each $v \notin U$ there is a $k_v \in \NN$ such that $\configg(v) = \config(v) + k_v$ and $\configg'(v) = \config'(v) + k_v$. 
    
  \item[activation operator {$\A[U],\A[v]$}] The operator $\A[U]$ activates any sleeping particles in $U \subseteq V$ without modifying the running odometer. That is $\A[U]\state = (\config',\odom)$ where $\config'$ differs from $\config$ by setting $\config'(v) = 1$ for every $v \in U$ such that $\config(v) = \s$. By convention $\A=\A[\V]$ and $\A[v]=\A[\{v\}]$. Note that acceptably toppling $v$ decomposes as activating and then legally toppling it. Thus, $\A[v]\state$ and $\Topple[v]\state = \Topple[v]\A[v]\state$ have the same stabilized state by the abelian property.
    
  \item[jump operator {$\Jump[v],\Jump,\Jump[\config],\Jump[\config,n]$}] For a configuration $\sigma$ containing
    a particle at $v$, the operator $\Jump[v]$ acts on $(\sigma,f)$ by performing (acceptable) topplings at $v$
    until executing a jump instruction. We call this operation \emph{jumping} a particle at $v$.
    The operator $\Jump$ jumps every particle in the configuration, i.e.,

    $$\Jump \state = (\Jump[v_n])^{|\config(v_n)|}(\Jump[v_{n-1}])^{|\config(v_{n-1}|)}\ldots (\Jump[v_1])^{|\config(v_1)|}\state$$
    where $\{v_1,\ldots,v_n\}$ is the support of $\config$. This operation does not depend on the order of $v_1,\ldots,v_n$ because it decomposes into a sequence
    of acceptable topplings that can be freely reordered by the local abelian property \cite[Property~1, Section~2.2]{rolla2020activated}. Also, note that after jumping every particle, there are no sleeping particles in the new configuration.
    
    Now let $\configg$ be a configuration of active particles. It is easy to show that the topplings in $\Jump \state$ are also acceptable starting in state $(\config + \configg, \odom)$, and lead to state $\Jump \state + (\configg, \zeroconfig)$, as the active particles of $\configg$ never move. Thus we define \emph{jumping $\config$}:$$\Jump[\config](\config + \configg, \odom) = \Jump \state + (\configg, \zeroconfig),$$ which decomposes as a sequences of acceptable topplings. Iterating then, jumping $\config$ multiple times still decomposes as an acceptable toppling sequence:
    $$\Jump[\config,n](\config + \configg, \odom) = \Jump^n \state + (\configg, \zeroconfig).$$

  \item[visiting a vertex] A state $\state$ \emph{visits $v$} if $v$ is toppled at least once
    during the stabilization of $\state$. Another equivalent formulation is that its running odometer at $v$ is under the stabilizing odometer, that is $\odom(v) < \trueodom(v)$ where $\trueodom$ is the odometer stabilizing $\state$.
    And another is that site~$v$ contains an active
    particle in one of the intermediate states (including the initial one) in the course of the stabilization
    of $\state$. A fourth is that $\config'(v) \ge 1$ where $\config'$ is the configuration of $\Stab[\{v\}^c] \state$. Conversely, $v$ is not visited by $\state$ if and only if $\config'(v) \in \{0, \s\}$.
    
    A subset $U \subseteq \V$ is visited by $\state$ if every site in $U$ is. 
    In keeping with the notation from the introduction, we say that $\sigma$
    visits $v$ or $U$ in place of $(\sigma,0)$ doing so.

\item[preemptive operations] Adding to the notion of a toppling being legal/acceptable, $\Topple[v]$ is \textit{preemptive} for $\state$ if it is acceptable and $\state$ visits $v$. 
We now have a three-tiered hierarchy of topplings: legal topplings are preemptive, which in turn are acceptable. Similarly, $\A[v]$ (resp.\ $\A[U]$) is preemptive for $\state$ if $\state$ visits $v$ (resp.\ $U$).
\end{description}

Finally now we construct a Markov process with the ARW dynamics using the sitewise
representation. Let the instructions $\Instr_v(j)$ be i.i.d.\ over all $v\in V$ and $j\geq 1$.
Each instruction takes the value \Sleep\ with probability $\lambda/(1+\lambda)$,
and otherwise is a jump instruction. Given that the instruction
at $v$ is a jump instruction, the distribution of its destination is given by $Q(v,\cdot)$, 
where $Q$ is the base chain.
Let the initial state be $(\sigma_0,0)$.
When the state is $(\sigma,f)$, each non-sink vertex~$v$ is toppled at rate given by the number
of active particles at $v$ in $\sigma$, at which time it transitions to state $\Topple[v](\sigma,f)$.
In our setting of finite graphs, this defines a finite-state continuous-time Markov chain.
On infinite graphs with reasonable initial configurations,  the Markov chain exists,
either by standard construction results for interacting particle systems or
by an explicit graphical construction (see \cite[Chapter~11]{rolla2020activated}).

\section{Bounds on general graphs}\label{sec:general}

We now have the notation to state and prove some new observations about ARW culminating with \thref{thm:one}. We begin with the preemptive abelian property, which states
that if a sleeping particle will be activated at some point in the system's stabilization, then we can activate it immediately without affecting the final stabilized state.

\begin{lemma}[Preemptive abelian property] \thlabel{lem:pap}
Suppose $v$ is acceptable to topple in state $\state$. If $(\sigma,f)$ visits $v$, then $(\sigma,f)$ has the same stabilized state
as $\Topple[v]\state$ and $\A[v]\state$. If $(\sigma,f)$ does not visit $v$, then the common stabilizing
odometer for $\Topple[v]\state$ and $\A[v]\state$ strictly dominates the odometer stabilizing $(\sigma,f)$.
\end{lemma}

\begin{proof}
    First, since $\Topple[v]\state \text{ has the same stabilized state as } \A[v]\state$, it suffices to prove the claim for $\A[v]\state =: (\config', \odom)$. Moreover, we only need to consider the case where $\config(v) = \s$. Otherwise, the claim is trivial: $\config' = \config$, and $\state$ visits $v$ since $\config(v) \ge 1$.

 Both $\state$ and $(\config', \odom)$ can be stabilized by first stabilizing $\{v\}^c$.
 This produces the same running odometer $\odomm$ since $\config$ and $\config'$ agree off of $v$.
 So let $$(\configg, \odomm) := \Stab[\{v\}^c]\state$$ and 
    $$(\configg', \odomm) := \Stab[\{v\}^c](\config', \odom),$$ which again have the same respective stabilized states as $\state, (\config', \odom)$. The configurations $\configg$ and $\configg'$ agree off of $v$, and $\configg(v) = \s + k$, $\configg'(v) = 1 + k$, where $k$ is the number of particles sent to $v$ when stabilizing $\state$ or $(\sigma',f)$
    on $\{v\}^c$.

     We claim that $k=0$ if and only if $\state$
     does not visit $v$. Indeed, if $k=0$ then $\configg(v)=\s$, implying that stabilizing $(\sigma,f)$ on $\{v\}^c$ also stabilizes it at $v$. Thus $(\sigma,f)$
     is stabilized without toppling $v$.
     Conversely, if $k\geq 1$ then a sequence of legal topplings starting from $(\sigma,f)$ 
     has sent a particle to $v$.
     
     Applying the claim, if $(\sigma,f)$ visits $v$, then $k\geq 1$. Therefore $\tau$ and $\tau'$ agree at $v$ and hence everywhere. Thus $(\configg, \odomm) = (\configg', \odomm)\text{ and indeed }(\config,\odom),(\config',\odom)\text{ have the same stabilized states}$. And if $\state$ does not visit $v$, then $k=0$ and $\configg$ is stable. Thus $\odomm$ is the odometer stabilizing $\state$. On the other hand $\configg'$ is not stable at $v$, so the odometer stabilizing $(\config', \odom)$ strictly dominates $\odomm$. 
\end{proof}

\begin{cor}\thlabel{cor:AU}
If activating $U$ is preemptive for $\state$, then $\A[U]\state$ has the same stabilized state as $\state$.
\end{cor}
\begin{proof}
    $\A[U]\state$ is performed by successively activating each $v \in U$. None of these activations modify the stabilizing odometer nor the running odometer $\odom$, so each one is preemptive.
\end{proof}

To explain why Lemma~3.1 is a preemptive form of the abelian property, note that it implies that a system can by stabilized by any sequence of preemptive topplings leaving it stable, a direct generalization of the abelian property from legal to preemptive topplings. An alternative perspective is that a sequence of acceptable topplings does not affect the stabilized state so long as none of the illegal topplings push the running odometer over the stabilizing odometer.

Next, we demonstrate an abelian property for jumps. Recall that when $\configg$ is a configuration of active particles, then starting from $\config + \configg$, \emph{jumping} $\config$ means toppling each particle in $\sigma$ until it jumps to a neighboring site.
If the non-jumped particles $\configg$ visit $\supp \config$, then jumping $\config$ does not affect the stabilized state:
\begin{lemma}\thlabel{lem:preemptive.jump}

Let $\configg$ be a configuration of active particles, and suppose that $(\configg, \odom)$ visits $\supp \config$. Then $\Jump[\config](\config + \configg, \odom)$ has the same stabilized state as  $(\config + \configg, \odom)$. 
\end{lemma}

\begin{proof}
Jumping $\config$ is done by executing a sequence of acceptable topplings. By the preemptive abelian property, it suffices to show none of these topplings increases the running odometer past the odometer stabilizing $\state$. At $v \in \supp \config$, jumping $\config$ increases the running odometer by $|\config(v)|$ jump instructions. On the other hand, $(\config + \configg, \odom)$ can be stabilized by first stabilizing $(\configg, \odom)$ on $\{v\}^c$.  By assumption, $k\ge 1$ particles do visit $v$, which now contains $|\config(v)| + k$ active particles. Jumping $|\config(v)|$ particles at $v$ at this point is legal, so these $|\config(v)|$ jump instructions are indeed all executed in the stabilization of $(\config + \configg, \odom)$.
\end{proof}

Now we prove the \emph{street-sweeper lemma} described in the introduction: if $\sweepersstart$ is a configuration of particles likely to visit all sites, then the stabilized configuration
of $\sweepersstart+\carsstart$ is nearly the same in law as the stabilized configuration
of $\sweepers$, since when stabilizing $\carsstart + \sweepersstart$ it is likely
we can topple the car particles off the graph preemptively.

\begin{lemma}[Street-sweeper lemma] \thlabel{lem:street-sweeper}
  Let $p:= \P(\sweepersstart \text{ does not visit all of $V$} )$. Let $N$ be the number of jumps required for configuration $\cars$ to reach the sink, i.e., $N\ge 0$ is the minimal value
  so that $\Jump[N](\cars,0)$ has an empty configuration. Then for any integer $m\geq 1$,
  \begin{align*}
  \tvnorm{\stablelaw{\cars + \sweepers} - \stablelaw{\sweepers}}&\leq mp + \P(N \ge m).
  \end{align*}
\end{lemma}

\begin{proof}
The idea is to set up a union bound over the particles in $\sweepers$ failing to preemptively topple those in $\cars$ over $m-1$ jumps. One's first thought might be that jumping the cars off the interval with the preemptive abelian property requires the street sweepers to visit each site at least $m-1$ times. A nice feature of the proof is that we use \thref{lem:preemptive.jump} to iteratively jump the car particles while the street sweepers remain in place. Since the street sweepers have not moved, we can use the Markov property of unrevealed instructions to recycle the probability of visiting all of the sites.

    Let $\sweepers'$ be the configuration of $\A( \sweepers,0)$, which consists solely of active particles. Let $$(\config_0, \odom_0) := (\cars + \sweepers', \zeroconfig) = \A[\supp \sweepers](\cars + \sweepers,0).$$ Activating $\supp \sweepers$ is preemptive for $(\cars + \sweepers,0)$ on the event $E_0 := \{\text{$(\sweepers,0)$ visits $V$}\}$, in which case $(\cars + \sweepers,0)$ has the same stabilized state as $(\config_0, \odom_0)$ by \thref{cor:AU}.

For $1 \le k \le m-1$, let $$(\config_k, \odom_k) = \Jump[\cars, k](\config_0,0) \text{ and } E_k := \{\text{$(\sweepers, \odom_k)$ visits $V$}\}.$$ Notice this agrees with the previously defined $E_0$. On $E_k$, $(\sweepers', \odom_k)$ also visits $V$, so $(\config_k,\odom_k)\text{ and }(\config_{k+1},\odom_{k+1})\text{ have the same stabilization}$ by \thref{lem:preemptive.jump}.  Therefore on $\bigcap_{k=0}^{m-2}E_k$, 
$(\cars + \sweepers,0)$ has the same stabilized state as $ (\config_{m-1}, \odom_{m-1})$. And on the event $\{N \le m-1\}$ we have $\config_{m-1} = \sweepers'$. Finally, on $E_{m-1}$, again by \thref{cor:AU} $(\sweepers', \odom_{m-1})$ and   $(\sweepers, \odom_{m-1})$ have the same stabilized state. Combining all of this, $(\cars + \sweepers,0)$ has the same stabilization as $(\sweepers, \odom_{m-1})$ on $\bigcap_{k=0}^{m-1}E_k \cap \{N \le m-1\}$.  

By the strong Markov property for quenched instructions \cite[Proposition~4]{levine2021source}, 
the stabilized configuration of $(\sweepers,\odom_{m-1})$ has law $\stablelaw{\sweepers}$, and each $E_k$ fails with the same probability $p$. Using this observation and a union bound 
over the complement $\bigcup_{k=0}^{m-1} E_k^c \cup \{N \geq m\}$ yields the lemma.
\end{proof}

\thref{thm:one} is an immediate corollary.
\begin{proof}[Proof of \thref{thm:one}]
    We take $\sweepers=\sigma$ and $\cars=\onepersite$.
    Then $\stationarysymb = \stablelaw{\onepersite + \config}$ by \cite[Theorem~1]{levine2021exact},
    and \thref{lem:street-sweeper} yields
    \begin{align*}
      \tvnorm{\stablelaw{\onepersite + \config} - \stablelaw{\config}}&\leq m\P(\text{$\sigma$ does not visit all sites}) + \P(H \ge m),
    \end{align*}
    noting that $H$ has the same distribution as $N$ with $\cars = \onepersite$ in the statement of the lemma.
\end{proof}

\section{Mixing in dimension one}\label{sec:dim1}

\thref{thm:one} implies that driven-dissipative ARW is mixed once we have added enough
particles to visit all sites. In this section, we complete the proof
of \thref{thm:mix} by showing in dimension one that adding a density
$\critFE+\epsilon$ of particles is enough. 
The proof relies on the following sufficient condition for a configuration
on $\ii{1,n}$ to emit particles from a given endpoint when stabilized.

\begin{prop}\thlabel{prop:exit.criterion}
  Let $\config$ be a configuration of active particles on $\ii{1,n}$
  containing at most $\beta n$ particles for some $\beta>0$.
  If
  \begin{align}
    \sum_{j=1}^n j\config(j)&\geq\frac{(\critFE+\epsilon)n^2}{2}\label{eq:right.exit}
  \end{align}
  for some $\epsilon>0$, then it holds with probability at least $1-Ce^{-cn}$
  that a particle exits $\ii{1,n}$ from the right endpoint when $\sigma$ is stabilized,
  where $c,C>0$ are constants depending only on $\lambda$, $\epsilon$, and $\beta$.
\end{prop}
\begin{remark}\thlabel{rmk:left.exit}
  By symmetry, if
  \begin{align}
    \sum_{j=1}^n (n-j+1)\config(j)&\geq\frac{(\critFE+\epsilon)n^2}{2},\label{eq:left.exit}
  \end{align}
  then
  it holds with probability  $1-Ce^{-cn}$ that
  a particle exits the interval from its left endpoint when $\sigma$ is stabilized.
\end{remark}
It is shown in \cite{forien2024macroscopic} that if $\sigma$ contains at least $(\critFE+\epsilon)n$
particles, a macroscopic quantity of them will exit $\ii{1,n}$ when $\sigma$ is stabilized
with positive probability. This result is strengthened in \cite{hoffman2024density}
to the statement that for any $\epsilon'<\epsilon$, at least $\epsilon'n$ particles 
exit $\ii{1,n}$ with probability exponentially close to $1$. 
The key aspect of \thref{prop:exit.criterion} compared to these results
is that it establishes that particles exit the interval via a specific endpoint.

The proposition includes the restriction that $\sigma$ contain at most $\beta n$ particles
for some fixed constant $\beta$, ruling out $\sigma$ consisting
of $(\critFE+\epsilon)n^2/2$ particles at site~$1$, for example.
We suspect that this restriction could be removed, but for this paper we are satisfied
with a less than optimal statement.

In the next section, we apply \thref{prop:exit.criterion} to prove \thref{thm:mix}.
Then in Section~\ref{sec:exit.proof} we prove \thref{prop:exit.criterion} using the tools
from \cite{hoffman2024density}. From now on, we work on the graph with vertices
$V=\ii{1,n}=\{1,\ldots,n\}$ with $1$ and $n$ connected to the sink,
and we assume that the base chain is simple symmetric
random walk.
We say that an event holds \emph{with overwhelming probability} (w.o.p.)\ if
its failure probability can be bounded by $Ce^{-cn}$ for constants $c,C>0$ that may depend
on some values to be specified but do not depend on $n$.

\subsection{Proof of Theorem~\ref{thm:mix} assuming Proposition~\ref{prop:exit.criterion}}

First, we use \thref{prop:exit.criterion} to prove that $(\critFE+\epsilon)n$
particles placed according to central or uniform driving are likely to visit every site in $\ii{1,n}$.

\begin{prop}\thlabel{prop:centralvisit}
  Let $\config$ consist of $\rho n$ active particles at site $\ceil{n/2}$ for $\rho>\critFE$.
  Then $\config$ visits all sites in $\ii{1,n}$ with probability
  at least $1-Ce^{-cn}$ for constants $c,C>0$ depending only on $\rho$ and $\lambda$.
\end{prop}
\begin{proof}
  \thref{prop:exit.criterion} immediately shows that with probability $1-Ce^{-cn}$,
  particles exit $\ii{1,n}$ from both endpoints when stabilizing $\config$, which implies that $\config$
  visits all sites.
\end{proof}

\begin{prop}\thlabel{prop:uniformvisit}
  Let $\config$ consist of $\rho n$ active particles for $\rho>\critFE$, each placed independently and
  uniformly at random in $\ii{1,n}$. Then $\config$ visits all sites with probability
  at least $1-Ce^{-cn}$ for constants $c,C>0$ depending only on $\rho$ and $\lambda$.
\end{prop}
\begin{proof}
  For $m\in\ii{n/2,n}$, consider the stabilization of $\config$ on the subinterval $\ii{1,m}$. Thinking of $\sum_{j=1}^mj\config(j)$
  as a sum over the particles in $\config$
  in which a particle at site $j$ contributes $j\1\{j\leq m\}$, we have
  \begin{align*}
    X:=\sum_{j=1}^mj\config(j) = \sum_{i=1}^{\rho n} Z_i,
  \end{align*}
  where $Z_1,\ldots,Z_{\rho n}$ are independent and
  \begin{align*}
    Z_i=\begin{cases}
      j &\text{with probability $1/n$ for $j\in\ii{1,m}$,}\\
      0 & \text{with probability $1-m/n$.}
    \end{cases}
  \end{align*}
  Since $X$ is a sum of independent random variables in $[0,m]$,
  \begin{align}\label{eq:S.conc}
    \P\bigl( X-\E X\leq -t \bigr)\leq \exp\biggl(-\frac{2t^2}{\rho n m^2}\biggr)
  \end{align}
  for any $t\geq 0$ by Hoeffding's inequality.
  Since $\E Z_i = m(m+1)/2n$, we have $\E X=\rho m(m+1)/2$.
  Defining $\epsilon=(\rho-\critFE)/2$ and applying \eqref{eq:S.conc} with
  $t = \epsilon m^2/2$, we obtain
  \begin{align*}
    \P\biggl(X\leq \frac{(\critFE+\epsilon)m^2}{2}\biggr) \leq
    \P\biggl(X\leq \frac{\rho m(m+1)}{2} - t\biggr)\leq \exp\biggl(-\frac{2\epsilon^2 m^4/4}{\rho n m^2}\biggr)
    \leq \exp\biggl(-\frac{\epsilon^2 n}{8\rho  }\biggr).
  \end{align*}
  Thus we have proven that condition~\eqref{eq:right.exit} holds
  with overwhelming probability. \thref{prop:exit.criterion} then implies that
  for any $m\geq n/2$, a particle exits $m$ when stabilizing $\config$ on $\ii{1,m}$ w.o.p.,
  with the constants in the bound allowed to depend on $\rho$ and $\lambda$.
  Thus $\config$ visits $m$ w.o.p. By symmetry, the same statement is true for $1\leq m<n/2$
  as well, and hence $\config$ visits all sites in $\ii{1,n}$ w.o.p.
\end{proof}

\begin{proof}[Proof of \thref{thm:mix}]
  Fix an initial configuration $\configstart$ on $V:=\ii{1,n}$.
  Suppose that $\mathbf{u}=(u_t)_{t\geq 1}$ is the uniform or central driving sequence.
  Let $\drive_t=\delta_{u_1}+\cdots+\delta_{u_t}$, the configuration
  consisting of the first $t$ driving particles, all taken as awake.
  By the abelian property, $\sigma_t$ is the stabilized configuration of $\sigma_0+\drive_t$.
  Since adding particles to a configuration can only increase the probability
  of visiting all sites,
  \begin{align*}
    \P(\text{$\sigma_0+\drive_{t}$ visits all sites}) \geq
    \P(\text{$\drive_t$ visits all sites}) \geq 1- Ce^{-cn}
  \end{align*}
  for $t> (\critFE+\epsilon)n$ by Proposition~\ref{prop:centralvisit} or \ref{prop:uniformvisit},
  where $c,C>0$ are constants depending only on $\epsilon$ and $\lambda$.
  Applying \thref{thm:one} with $m=n^3$ and using standard
  random walk bounds to get an exponential bound on $\P(H\geq n^3)$,
  we prove \eqref{eq:mixing.upper}.
  
  For the lower bound on mixing time, let $\sigma_0$ be the empty configuration.
  By \cite[Proposition~8.5]{hoffman2024density},
  the density in a sample from $\pi$ is at least $\critFE-\epsilon$ with probability
  at least $1 - Ce^{-cn}$ for some $c,C$ depending only on $\lambda$ and $\epsilon$. 
  On the other hand, $\sigma_t$ has probability~$0$ of having density at least $\critFE-\epsilon$
  when $t<(\critFE-\epsilon)n$, establishing \eqref{eq:mixing.lower}.
\end{proof}

\subsection{Proof of Proposition~\ref{prop:exit.criterion}}\label{sec:exit.proof}

Let $f\colon\ZZ\to \NN$ be the stabilizing
odometer for an initial configuration $\sigma$ on $\ii{1,n}$, taking $f(k)=0$
for $k\notin\ii{1,n}$
(we will often work on $\ZZ$ rather than $\ii{1,n}\cup\{z\}$, with $0$ and $n+1$
identified with the sink $z$).
Considering the version of \thref{prop:exit.criterion} given
in \thref{rmk:left.exit}, our goal is to show that if $\sigma$ satisfies \eqref{eq:left.exit},
then $f(1)$ is likely to be large.
The odometer $f$ is a \emph{stable odometer} on $\ii{1,n}$
in the sense of \cite[Definition~2.2]{hoffman2024density},
meaning that executing topplings according to $f$ leaves a stable configuration on $\ii{1,n}$.
Using the theory developed in \cite{hoffman2024density}, all stable odometers
can be embedded into a process called \emph{layer percolation}. Via this correspondence, 
we will prove that \emph{all} stable odometers are large at site~$1$.
(This statement seems stronger than $f(1)$ being large, but it is in fact equivalent
by the \emph{least-action principle} \cite[Lemma~2.3]{hoffman2024density}, which states that
the stabilizing odometer on a set of sites is the minimal odometer stable on those sites.)

To sketch the proof, we first need to describe layer percolation and its
correspondence with ARW. For a more complete introduction with many examples,
see \cite[Section~3]{hoffman2024density}.
\emph{Layer percolation with parameter $\lambda>0$}
is a form of $(2+1)$-dimensional directed percolation in which two-dimensional \emph{cells}
in one step infect cells
in the next. Cells take the form $(r,s)_k$, where $r$ is the \emph{column},
$s$ is the \emph{row}, and $k$ is the step. We write $(r,s)_k\to(r',s')_{k+1}$ to denote that
$(r,s)_k$ infects $(r',s')_{k+1}$. Layer percolation can be coupled with ARW with sleep
rate $\lambda$ so that stable odometers are embedded in layer percolation as
\emph{infection paths}, chains of cells each infecting the next.

To state the correspondence between ARW and layer percolation, we first extend the
instruction stacks $\Instr_v(j)$ to allow $j\leq 0$. Since we are working exclusively
in dimension one, we indicate the jump instructions by \Left\ and \Right.
We then consider \emph{extended odometers}, which are permitted to take negative values.
When an odometer takes a negative value $-k$ at site~$v$, it represents execution of instructions
$\Instr_v(-k+1),\ldots,\Instr_v(0)$, with each instruction acting in reverse (e.g., a \Left\ instruction
at site~$v$ snatches a particle from $v-1$ rather than pushing a particle there).
The definition
of stable odometers carries over to extended odometers as well, though we leave
all these details for \cite[Section~4]{hoffman2024density} since they will not matter here.
For an extended odometer $f$, we define $\lt{f}{v}$ and $\rt{f}{v}$ as quantities representing
the number of \Left\ and \Right\ instructions executed by $f$ at $v$. When $f(v)\geq 0$,
these quantities are just the counts of the number of \Left\ and \Right\ instructions,
respectively, in $\Instr_v(1),\ldots,\Instr_v(u(v))$.
We define $\eosos{n}=\eosos{n}(\Instr,\config,u_0,f_0)$ as the set of extended odometers on $\ii{0,n}$
that are stable on $\ii{1,n-1}$ for initial configuration $\sigma$,
take the value $u_0$ at $0$, and yield a net flow $f_0$ from site~$0$ to site~$1$
(i.e., for $f\in\eosos{n}(\Instr,\config,u_0,f_0)$ we have
$\rt{f}{0}-\lt{f}{1}=f_0$).
The set $\eosos{n}$ contains a minimal element we call the \emph{minimal odometer},
though typically it is only an extended odometer.
We define $\ip{n}=\ip{n}(\Instr,\sigma,u_0,f_0)$ as the set of infection paths of length~$n$
starting from $(0,0)_0$ in a realization of layer percolation constructed
from $\Instr$ depending on $\sigma$, $u_0$, and $f_0$ (see \cite[Definition~4.4]{hoffman2024density}
for the construction). The sets $\eosos{n}$
and $\ip{n}$ are nearly in bijection. 
When an odometer is represented as an infection
path $(0,0)_0=(r_0,s_0)_0\to\cdots\to(r_n,s_n)_n$, each $r_k$ reflects the number of \Right\ instructions
executed by the odometer at site~$k$, and each $s_k$ counts the number of particles left sleeping
on $\ii{1,k}$ by the odometer:
\begin{prop}[{\cite[Proposition~4.6]{hoffman2024density}}]\thlabel{bethlehem}
For an extended odometer $u$, define $\Phi(u)$ as the sequence of cells
 $\bigl((r_v,s_v)_v,\,0\leq v\leq n\bigr)$ given by
\begin{align*}
  r_v &= \rt{u}{v} - \rt{\mo}{v},\\
  s_v &= \sum_{i=1}^v\1\{\Instr_i(u(i))=\Sleep \},
\end{align*}
where $\mo$ is the minimal odometer of $\eosos{n}(\Instr,\sigma,u_0,f_0)$.
Then $\Phi$ is a surjection from $\eosos{n}(\Instr,\sigma,u_0,f_0)$
onto $\ip{n}(\Instr,\sigma,u_0,f_0)$. If $u,u'\in\eosos{n}(\Instr,\sigma,u_0,f_0)$
are two extended odometers with $\Phi(u)=\Phi(u')$, then for all $v\in\ii{0,n}$,
either $u(v)=u'(v)$
or $u(v)$ and $u'(v)$ are two indices in $\Instr_v$ in a string of consecutive \Sleep\ instructions.
\end{prop}

Let $S_n$ be the largest row in a realization of layer percolation that contains a cell at step~$n$
infected starting from $(0,0)_0$. As a consequence of subadditivity we have
$S_n/n\to\crist$ in probability, where
$\crist=\crist(\lambda)\in[0,1]$ is a constant \cite[Proposition~5.18]{hoffman2024density}.
Eventually it is shown that $\crist=\critFE=\critDD$ \cite[Theorem~1]{hoffman2024density}.
We will generally refer to $\crist$ rather than $\critFE$ for the rest of the argument
to avoid confusion when we cite results from \cite{hoffman2024density} that make reference
to $\crist$.

To prove that all stable odometers on $\ii{1,n}$ execute at least one \Left\ instructions
at site~$1$, we must show that all the infection paths in certain sets
$\ip{n}(\Instr,\sigma,u_0,f_0)$ do not correspond to odometers;
that is, all the elements of the corresponding $\eosos{n}$ are 
only extended odometers because they take negative values at some site.
We will show specifically that the odometers are all negative at site~$n$.
From the definition of $\Phi$ in \thref{bethlehem}, 
this fact comes down to the rightmost
column infected at step~$n$ of layer percolation being bounded by $\rt{\mo}{n}$, the number
of \Right\ instructions executed at site~$n$ by the minimal odometer.
Thus our first step is to bound the rightmost column infected in layer percolation.

\begin{lemma}\thlabel{lem:column.boundary}
  Let $R_n$ be the rightmost column infected at step~$n$ starting from
  cell $(0,0)_0$ in layer percolation with sleep parameter $\lambda>0$.
  For any $\epsilon>0$,
  \begin{align*}
    \P\biggl(R_n \geq \frac{(\crist+\epsilon)n^2}{2}\biggr) \leq Ce^{-cn},
  \end{align*}
  where $c,C>0$ are constants depending only on $\lambda$ and $\epsilon$.
\end{lemma}
\begin{proof}
  \newcommand{\rowboundevent}{\mathsf{RowBound}}\newcommand{\colboundevent}{\mathsf{ColBound}}
  The idea of the proof is to take the bound \cite[Proposition~7.1]{hoffman2024density}
  on the topmost row infected and apply \cite[Lemma~5.5]{hoffman2024density} to transfer the
  bound from row to column. We allow the constants in w.o.p.\ bounds to depend on
  $\lambda$ and $\epsilon$.
  Let $\epsilon' = \epsilon/8$.
  Let $\smax_j = \crist j + \epsilon'n$.
  We say that an infection path $(r_0,s_0)_0\to\cdots\to(r_n,s_n)_n$ is \emph{row-bounded}
  if $s_j\leq\smax_j$ for each $0\leq j\leq n$.
  Define $\rowboundevent$ as the event that every infection path from $(0,0)_0$
  is row-bounded.
  And we define $\colboundevent$ as the event that every row-bounded infection path
  $(0,0)_0=(r_0,s_0)_0\to\cdots\to(r_n,s_n)_n$ satisfies $r_n<(\crist+\epsilon)n^2/2$.
  We aim to show that $\rowboundevent$ and $\colboundevent$ occur with overwhelming
  probability, which proves the lemma.
  
  To show that $\rowboundevent$ occurs with overwhelming probability, let
  $S_j$ be the maximum of $s_j$ over all infection paths
  $(0,0)_0=(r_0,s_0)_0\to\cdots\to(r_n,s_n)_n$.
  Since infection paths move up by at most one row in each step,
  for all $j\leq\epsilon' n$ we have $S_j\leq j\leq \smax_j$.
  The bound \cite[Proposition~7.1]{hoffman2024density} gives
  \begin{align*}
    \P\bigl(S_j\geq \smax_j\bigr)\leq \P\bigl(S_j\geq (\crist+\epsilon')j\bigr)
      \leq C e^{-cj}
  \end{align*}
  where $c,C>0$ depend only on $\epsilon'$ and $\lambda$. Taking a union
  bound, it holds with overwhelming probability
  that $S_j\leq\smax_j$ for $\epsilon' n<j\leq n$. Hence $\rowboundevent$
  occurs with overwhelming probability.
  
  To bound the probability of $\colboundevent$, we note that by \cite[Lemma~5.5]{hoffman2024density},
  the column of any row-bounded infection path at step~$j$ is bounded by $r_j$
  where $r_0=0$ and $(r_j+\smax_j+1)_{j=0}^n$ is a Galton--Watson process with child distribution $\Geo(1/2)$
  and immigration $\smax_j+1$ at step~$j$ for $j\geq 1$. By \cite[Proposition~A.1]{hoffman2024density},
  it holds with overwhelming probability that $r_n+\smax_n+1$ is within $\epsilon' n^2/2$
  of its mean, which is
  \begin{align*}
    \sum_{j=0}^n(\smax_j+1) = \frac{\crist n^2}{2} + \epsilon'n^2 + O(n).
  \end{align*}
  Thus $r_n$ is within $\epsilon' n^2$ of $\crist n^2/2+\epsilon' n^2$ w.o.p.
  Therefore $r_n\leq \crist n^2+2\epsilon'n^2<(\crist+\epsilon)n^2/2$ w.o.p., 
  thus proving that $\colboundevent$ holds w.o.p.
\end{proof}

We are ready to prove \thref{prop:exit.criterion} now.
Essentially, the result is a consequence of the previous
lemma together with \thref{bethlehem} and
the concentration estimate \cite[Proposition~5.8]{hoffman2024density},
which pinpoints the value of $\rt{\mo}{k}$ for the minimal odometer $\mo$
up to order $n^2$.

\begin{proof}[Proof of \thref{prop:exit.criterion}]
  We will prove the version of the proposition from \thref{rmk:left.exit},
  which by symmetry
  proves the original version as well. Thus we assume \eqref{eq:left.exit}
  and prove that a particle moves from site~$1$ to the sink at site~$0$ w.o.p., with
  constants in w.o.p.\ bounds
  allowed to depend on $\lambda$, $\epsilon$, and $\beta$.

    Let $u_0=f_0=0$, and let $\Instr$ denote the instructions for ARW with
  sleep rate $\lambda>0$.
  Consider the class $\eosos{n}=\eosos{n}(\Instr,\config,u_0,f_0)$, the set
  of extended odometers $u$ on $\ii{0,n}$ stable on $\ii{1,n-1}$
  satisfying $u(0)=0$ and $\lt{u}{1}=0$. If no particle exits $\ii{1,n}$ at the left
  endpoint, then no \Left\ instruction is executed at site~$1$ and thus the odometer
  stabilizing $\config$ falls into $\eosos{n}$. Thus, it suffices to show
  that $\eosos{n}$ contains no odometers w.o.p., 
  i.e., every extended odometer in $\eosos{n}$
  takes negative values. We will show specifically that every $u\in\eosos{n}$ satisfies $u(n)<0$ w.o.p.
  
    Let $\m$ be the minimal odometer of $\eosos{n}$.
  By \cite[Proposition~5.8]{hoffman2024density}, it holds with overwhelming
  probability that $\rt{\m}{n}$ is within $\epsilon n^2/4$ of
  \begin{align*}
    -\sum_{i=1}^n \sum_{v=1}^i\config(v)=-\sum_{j=1}^n (n-j+1)\config(j)
      \leq-\frac{(\crist+\epsilon)n^2}{2}.
  \end{align*}
  Here we are using our assumption that $\config$ contains at most $\beta n$ particles
  to make the right-hand side of \cite[eq.~(24)]{hoffman2024density}
  exponentially small for $t=\epsilon n^2/2$ by making $\emax\leq \beta n$.
  Hence
  \begin{align}\label{eq:min.bound}
    \rt{\m}{n}\leq -(\crist+\epsilon')n^2/2\text{ w.o.p.,}
  \end{align}
  where $\epsilon'=\epsilon/2$.

  Let $\ip{n}=\ip{n}(\Instr,\config,u_0,f_0)$, the set of length~$n$ infection paths
  starting from $(0,0)_0$ in layer percolation
  generated from $\Instr$. By \thref{bethlehem}, each extended
  odometer $u\in\eosos{n}$ corresponds to an infection path
  \begin{align}
    (r_0,s_0)_0\to\cdots\to(r_n,s_n)_n\label{eq:ip}
  \end{align}
  in $\ip{n}$ with $\rt{u}{n}=r_n+\rt{\m}{n}$.
  By \thref{lem:column.boundary}, it holds with overwhelming probability
  that $r_n< (\crist+\epsilon')n^2/2$ for all infection paths \eqref{eq:ip}
  in $\ip{n}$. Together with \eqref{eq:min.bound}, this proves that $\rt{u}{n}<0$
  and hence that $u(n)<0$ for all $u\in\eosos{n}$ w.o.p. Hence all extended
  odometers in $\eosos{n}$ take a negative value. Thus the stabilizing
  odometer is not found in $\eosos{n}$, and we can conclude that it executes at least one
  \Left\ instruction at site~$1$ w.o.p.
\end{proof}

\begin{comment}
\HOX{Is the plan here to do a general ``transfer of high-probability events'' theorem
or to go directly to uniform density? We've discussed this on Discord
and I'm not sure we made a final decision.--TJ}

\begin{proof}[Proof of \thref{thm:density}]
    Let $I = \ii{a,b}$. The idea is to stabilize $\ii{1,n}$ by successively stabilizing $I$ and $I^c$. By the density conjecture, each time $I$ is stabilized a near-critical density of sleeping particles is left on $I$.
    \begin{enumerate}[(i)]
        \item 
    \end{enumerate}
\end{proof}
\end{comment}

\section*{Acknowledgments}

Hoffman was partially supported by NSF grant DMS-1954059.
Junge was partially supported by NSF grant DMS-2238272.
We thank Leo Rolla for helpful conversations.
We are grateful to BIRS, which hosted the workshop \emph{Markov Chains with Kinetic Constraints and Applications}.

\bibliographystyle{amsalpha}
\bibliography{main}

\end{document}